\documentclass[oneside,english,a4wide]{amsart}
\usepackage[latin9]{inputenc}
\usepackage[a4paper]{geometry}
\usepackage{mathabx}
\usepackage{esvect}
\usepackage{babel}
\usepackage{mathtools}
\usepackage{amstext}
\usepackage{amsthm}
\usepackage{amssymb}
\usepackage{enumerate}
\usepackage{esint}
\usepackage{graphicx}
\usepackage{bbm}
\usepackage[all]{xy}
\usepackage{mathrsfs}
\usepackage[unicode=true,
bookmarks=true,bookmarksnumbered=true,bookmarksopen=true,bookmarksopenlevel=2,
breaklinks=false,pdfborder={0 0 1},backref=false,colorlinks=false]
{hyperref}
\hypersetup{
	colorlinks,linkcolor=blue,anchorcolor=blue,citecolor=blue}
\usepackage{breakurl}
\usepackage{autobreak}

\usepackage[svgnames]{xcolor} 
\usepackage{graphicx} 

\makeatletter
\numberwithin{equation}{section}
\numberwithin{figure}{section}
\theoremstyle{plain}
\newtheorem{thm}{\protect\theoremname}[section]
\theoremstyle{plain}

\newtheorem{proposition}[thm]{Proposition}
\newtheorem{corollary}[thm]{Corollary}
\theoremstyle{definition}

\theoremstyle{plain}

\newtheorem{lemma}[thm]{Lemma}
\theoremstyle{remark}
\newtheorem{rem}[thm]{\protect\remarkname}

\theoremstyle{definition}

\newtheorem*{claim*}{Claim}

\theoremstyle{remark}

\theoremstyle{definition}
\newtheorem*{defn*}{\protect\definitionname}

\usepackage{babel}
\usepackage{amscd}

\makeatother

\providecommand{\definitionname}{Definition}
\providecommand{\lemmaname}{Lemma}
\providecommand{\propositionname}{Proposition}
\providecommand{\remarkname}{Remark}
\providecommand{\theoremname}{Theorem}

\newcommand{\Rmnum}[1]{\expandafter\@slowromancap\romannumeral #1@}
\makeatother

\newcommand{\real}{{\mathbb R}}
\newcommand{\nat}{{\mathbb N}}

\newcommand{\com}{{\mathbb C}}

\newcommand{\e}{\varepsilon}

\newcommand{\Tr}{\mbox{\rm Tr}}

\newcommand{\8}{\infty}

\newcommand{\wt}{\widetilde}

\newcommand{\be}{\begin{eqnarray*}}
	\newcommand{\ee}{\end{eqnarray*}}
\newcommand{\beq}{\begin{equation}}
	\newcommand{\eeq}{\end{equation}}
\newcommand{\beqn}{\begin{equation*}}
	\newcommand{\eeqn}{\end{equation*}}
\newcommand{\bs}{\begin{split}}
	\newcommand{\es}{\end{split}}

\newcommand\norm[1]{ \left\| #1 \right\| }
\newcommand\jdz[1]{ \left| #1 \right| }
\newcommand\sk[1]{ \left( #1 \right) }
\newcommand\mk[1]{ \left[ #1 \right] }
\newcommand\lk[1]{ \left\{ #1 \right\} }

\renewcommand\d{{\rm d}}

\begin{document}

	
	
	\title[Hypercontractivity for quantum Ornstein-Uhlenbeck semigroups]{Hypercontractivity for a family of quantum Ornstein-Uhlenbeck semigroups}
	
	
	\thanks{Longfa Sun is supported by the National Natural Science Foundation of China (Grant No. 12101234) and the Fundamental Research Funds for Central Universities (Grant No. 2025MS178). Zhendong Xu is partially supported by the National Research Foundation of Korea (NRF) grant funded by the Korea government (MSIT) (No.2020R1C1C1A01009681). Zhendong Xu and Hao Zhang are partially supported by Samsung Science and Technology Foundation under Project Number SSTF-BA2002-01. Zhendong Xu and Hao Zhang are partially supported by the Basic Science Research Program through the National Research Foundation of Korea (NRF) (Grant RS-2022-NR069971).}

	\author[Longfa Sun]{Longfa Sun}
	\address{Hebei key laboratory of physics and energy technology, School of mathematics and physics, North China Electric Power University, Baoding 071003, China}
	\email{sun.longfa@ncepu.edu.cn}
	
	\author[Zhendong  Xu]{Zhendong Xu}
	\address{
		Department of Mathematical Sciences and the Research Institute of Mathematics, Seoul National University, Gwanak-ro 1, Gwanak-gu, Seoul 08826, Republic of Korea}
	\email{zhendong\_xu\_97@snu.ac.kr}

	\author[Hao Zhang]{Hao Zhang}
	\address{
		Instituto de Ciencias Matem\'{a}ticas, Consejo Superior de Investigaciones Cient\'{i}ficas, C/ Nicol\'{a}s Cabrera 13-15, 28049, Madrid, Spain}
	\email{hao.zhang@icmat.es}
	\date{}
	\maketitle
	
	\begin{abstract}
		We show that a family of quantum Ornstein-Uhlenbeck semigroups is hypercontractive. We also obtain the optimal order of the optimal time up to a constant. The main ingredient of our proof is Meixner polynomials. The goal of this paper is twofold: to provide more examples of hypercontractive quantum Markov semigourps on non-tracial von Neumann algebras, and to determine the optimal order of the optimal time for quantum Ornstein-Uhlenbeck semigroups.
	\end{abstract}
	
	
	\section{Introduction}\label{Introduction}
	In the celebrated papers \cite{EN1966,EN1973}, Nelson showed that the classical Ornstein-Uhlenbeck semigroup is hypercontractive and utilized such property to establish the existence and uniqueness of the ground state of a special semigroup which arises from the constructive quantum field theory. Since then, hypercontractivity became a useful and powerful tool in quantum field theory, quantum statistical mechanics and quantum information. We would like also to remark that the first example of optimal hypercontractivity is the ``Poisson semigroup" in the boolean setting, which was established in the celebrated works of Bonami \cite{B1968, B1970} by the famous two-point inequality.
	
	Gross discovered the equivalence of log-Sobolev inequalities and hypercontractivity, and also gave a new proof of the hypercontractivity for the classical Ornstein-Uhlenbeck semigroup in the remarkable work \cite{LG1975}. This surprising observation led to the application of hypercontractivity in numerous fields, such as Boolean analysis, concentration inequalities, geometric inequalities, statistical physics, among others. We refer to \cite{LG1993,BGL2013} for more information.
	
	Quantum Markov semigroups, which are a generalization of classical Markov semigroups, are a fundamental tool to describe open quantum systems. In \cite{CFL2000}, Cipriani \emph{et al.} rigorously proved that a specific unbounded Lindblad-type operator generates a quantum Markov semigroup by virtue of the remarkable quantum semigroup theory established in \cite{C1992,C1997}. Such unbounded Lindblad-type operator has been extensively studied in quantum optics models of masers and lasers \cite{D1976,FRS1994}. The quantum Markov semigroup appearing in \cite{CFL2000}, which is now called quantum Ornstein-Uhlenbeck semigroup, has also been shown to be hypercontractive in the seminal work \cite{CS2008}.
	
	The need to construct quantum Markov semigroups on von Neumann algebras, which are symmetric with respect to a nontracial state, is clear for various applications to open quantum systems, quantum statistical mechanics and quantum information (see \cite{KP} for more information). In the remarkable paper \cite{KP}, Ko \emph{et al.} constructed a family of quantum Ornstein-Uhlenbeck semigroups, which can be seen as a generalization of the quantum Ornstein-Uhlenbeck semigroup appearing in \cite{CFL2000, D1976,FRS1994}.
	
	In this paper, we aim to study hypercontractivities of the quantum Ornstein-Uhlenbeck semigroups established in \cite{KP}. At first, we introduce such quantum Ornstein-Uhlenbeck semigroups. To this end, we present Weyl operators and CCR. 
	
	Let $H=\ell_2(\mathbb{N})$, where $\mathbb{N}=\{0,1,2, \cdots\}$ is the set of all natural numbers. Let $\{e_n\}_{n=0}^\infty$ be the canonical orthonormal basis of $H$. Then we define the creation and annihilation operators $a^*$ and $a$ as follows
	\[
	a^* e_n = \sqrt{n+1} \, e_{n+1}, \quad a e_n = \sqrt{n} \, e_{n-1}, \quad (e_{-1} := 0).
	\]
	It is clear that they are densely defined closable and satisfy the canonical commutation relation (CCR)
	\[
	[a, a^*] = 1,
	\]
	where $1$ denotes the identity operator by a slight abuse of notation. The number operator $N := a^* a$ acts as
	\[
	N e_n = n \cdot e_n.
	\]
	The position and momentum operators are defined as
	\[
	Q = \frac{1}{\sqrt{2}}(a + a^*), \quad P = \frac{1}{\sqrt{2}i}(a - a^*),
	\]
	and so
	\[
	[Q, P] = i \cdot 1.
	\]
	For any $z \in \mathbb{C}$, define the Weyl operators
	\[
	W(z) = e^{\frac {i}{\sqrt{2}}(z a^* + \bar{z} a)} = e^{i ( \Re z \cdot Q + \Im z \cdot P ) }.
	\]
	These operators satisfy the Weyl relation
	\[
	W(z)^*=W(-z), \quad \text{and} \quad \forall z, w\in\mathbb{C}, \ W(z) W(w) = e^{-\frac{i}{2} \Im (\bar{z} w)} W(z + w).
	\]
	Let $\mathcal{A}$ be the $C^*$-algebra generated by all Weyl operators. It is well-known that $\mathcal{A}\subsetneqq B(H)$ and $\mathcal{A}{''}=B(H)$. Let $\omega$ be the normal faithful state on $B(H)$, which is defined by
	$$ \omega(W(z))=e^{-\frac{|z|^2}{4} (1+e^{-\beta})(1-e^{-\beta})^{-1}}, \quad \forall \, z\in \mathbb{C} , $$
	where $\beta>0$ is a fixed inverse temperature. Note that $e^{-\beta N}$ is a trace class operator in $B(H)$. Let 
	$$ \rho=\dfrac{e^{-\beta N}}{\text{Tr}(e^{-\beta N})}=(1-e^{-\beta})e^{-\beta N}. $$
	Indeed, $\omega$ is a Gibbs state, and for any $x\in B(H)$,
	$$  \omega(x)=\text{Tr}(\rho x).  $$
	Moreover, the group $(\sigma_t)_{t\in\mathbb{R}}$ of modular automorphisms on $B(H)$ associated with the Gibbs state $\omega$ is given by
	$$ \sigma_t(W(z))= W(e^{-i\beta t}z), \ \forall z\in \mathbb{C}.   $$
	We refer the interested reader to \cite{BR} for more details. 
	
	Now we introduce the quantum Ornstein-Uhlenbeck semigroups constructed in \cite{KP}. Let the parameters $\alpha_1, \alpha_2, \alpha_3 \in \mathbb{R}$ satisfy the following relations
	\begin{align*}
		& \frac{1}{2}\left(1+\alpha_2^2\right) \sinh (\beta / 2)=-\alpha_1 \cosh (\beta / 2), \\
		& \frac{1}{2}\left(\alpha_1^2+\alpha_3^2\right) \sinh (\beta / 2)=\alpha_2 \alpha_3 \cosh (\beta / 2).
	\end{align*}
	Let $G$ be the elliptic operator on $B(H)$ given by
	\begin{equation*}
		\begin{aligned}
			G(A)=\frac{\gamma}{2}\left(1+\alpha_2^2\right)[P,[P, A]]&+\frac{\gamma}{2}\left(\alpha_1^2+\alpha_3^2\right)[Q,[Q, A]]-i \gamma \alpha_1(Q[P, A]+[P, A] Q)\\
			&-i \gamma \alpha_2 \alpha_3(P[Q, A]+[Q, A] P), \quad \quad \forall \, A\in B(H),
		\end{aligned}
	\end{equation*}
	where $\gamma=\left(1+\alpha_2^2\right)^{-1}$ is the normalized constant. Since $P$ and $Q$ are unbounded operators affiliated to $\mathcal{M}$, the above is a formal expression. Indeed, the domain of $G$ is $\sigma$-weak dense in $B(H)$, denoted by ${\rm dom}(G)$; see \cite{KP} for the detailed definition of $G$. For $0<p\leq \8$, let $S_p$ be the Schatten $p$-class of $B(H)$ with the usual Schatten $p$-norm denoted by $\|\cdot\|_p$. In addition, $B(H)$ is often written as $S_\infty$.
	Define $\mathcal{G}^{(2)}$ on $S_2$ by
	$$  \mathcal{G}^{(2)}(\rho^{1/4}x\rho^{1/4})=\rho^{1/4}G(x)\rho^{1/4}, \quad \forall \, x\in  {\rm dom}(G).  $$
	According to \cite{KP}, $\mathcal{G}^{(2)}$ is a positive self-adjoint operator on a dense domain in $S_2$ and $P_t= e^{-t\mathcal{G}^{(2)}}$ ($\forall \, t\geq 0$) is a symmetric semigroup on $S_2$. For any $x\in B(H)$, define
	$$  \rho^{1/4} T_t^{(\8)}(x) \rho^{1/4}= P_t(\rho^{1/4}x\rho^{1/4}).  $$
	Then $T_t^{(\8)}$ is an ergodic quantum Markov semigroup on $S_\8$, and $\omega$ is the unique invariant state associated with $T_t^{(\8)}$. We would like to emphasize that in general, each $T_t^{(\8)}$ does not commute with the group of modular automorphism $(\sigma_t)_{t \in \real}$. In addition, the restriction of $T_t^{(\8)}$ to the maximal abelian subalgebra generated by $\{ \exp(ikQ), k \in \real\}$ is reduced to a unique classical Ornstein-Uhlenbeck semigroup. See \cite{KP} for more details. 
	
	In order to present the definition of hypercontractivity, noncommutative $L_p$ spaces need be introduced. Since the Gibbs state $\omega$ is induced by the density operator $\rho$, we use Kosaki's definition \cite{Ko1984} for noncommutative $L_p$ spaces, instead of Haagerup's noncommutative $L_p$ spaces for simplicity. We would like to stress that Kosaki's noncommutative $L_p$ spaces are now becoming popular in the study of quantum information theory, see for instance \cite{B2013, MDSFT2013}. In this paper, Kosaki's noncommutative $L_p$ spaces, denoted by $L_p(\rho)$ for $1\leq p<\8$, are defined as the closure of all elements in $\rho^{1/2}B(H)\rho^{1/2}$ with respect to the following norm
	$$  \| \rho^{1/2} x \rho^{1/2} \|_{L_p(\rho)}:= \| \rho^{1/2p} x \rho^{1/2p}\|_p, \quad  \forall \, x\in B(H). $$
	
	For $1\leq p<\8$, denote by $T_t^{(p)}: L_p(\rho)\rightarrow L_p(\rho)$ the usual restriction of $T_t^{(\8)}$ to $L_p(\rho)$, by using the same way as in \cite[Section 2]{GL1995} or \cite[Section 5]{HMX2009}. In this paper, we are only concerned with the case $p=2$, and so we just give the explicit formula of $T_t^{(2)}$. More explicitly, for $p=2$, if $x\in B(H)$, then
	$$  T_t^{(2)}(\rho^{1/2}x\rho^{1/2})=\rho^{1/4}P_t(\rho^{1/4}x\rho^{1/4})\rho^{1/4}. $$
	Denote by $\tau$ the spectral gap of $T_t^{(2)}$, which is defined in \eqref{gap}.
	
	We present the definition of hypercontractivity. Let $(T_t)_{t\geq 0}$ (also denoted by $(T_t^{(\infty)})_{t\geq 0}$) be an ergodic quantum Markov semigroup on $S_\infty$ with the invariant state induced by $\rho$. Then $(T_t)_{t\geq 0}$ is said to be \emph{hypercontractive} if for any $2<p<\8$, there exists $t_p>0$ such that for any $t \geq t_p$, and for any $ x \in L_2(\rho) $,
	\begin{equation}\label{hypercontractivity}
		\| T_t^{(2)} (x) \|_{L_p(\rho)} \leq \|x\|_{L_2(\rho)}.
	\end{equation}
	Moreover, the least $t_p$ is called the optimal time of $T_t$ with respect to the parameter $p$. We would like to remark that our definition of hypercontractivity is indeed consistent with that in \cite[Definition 3.1]{OZ1999}, which is the first paper that attempts to define hypercontractivity in the nontracial setting.
	
	Our main theorem is the following:
		\begin{thm}\label{sub-MainThm}
		For any $2<p<\infty$, there exists $t_p>0$ such that for any $t \geq t_p$, and for any $ x\in L_2(\rho) $,
		\begin{equation*}
			\| T_t^{(2)}(x) \|_{L_p(\rho)} \leq \|x\|_{L_2(\rho)}.
		\end{equation*}
		Moreover, the optimal time $t_p$ for all $ x\in L_2(\rho) $ satisfies
		\[
		\wt{c}(\beta) (p - 1) \leq e^{2\tau t_p} \leq \wt{C}(\beta) (p - 1), \quad \forall\, 2 < p < \infty,
		\]
		where $\wt{c}(\beta)$ and $\wt{C}(\beta)$ are positive constants depending only on $\beta$.
	\end{thm}

	Several comments are in order. Theorem \ref{sub-MainThm} implies that $(p-1)^{1/2}$ is the optimal order of $e^{\tau t_p}$ up to a constant depending only on $\beta$. Our proof of Theorem \ref{sub-MainThm} relies heavily on Meixner polynomials. To the best of our knowledge, this is perhaps the first time that Meixner polynomials are used to study quantum Ornstein-Uhlenbeck semigroups. Therefore, the proof of Theorem \ref{sub-MainThm} is different from the previous ones and completely new. Throughout this paper, we do not care about the constants $\wt{c}(\beta)$ and $\wt{C}(\beta)$ as we are only concerned with the order of $(p-1)$ as $p\rightarrow \infty$.
	
	In the noncommutative setting, hypercontractivity has been studied for several other types of Ornstein-Uhlenbeck semigroups; see \cite{CL1993,BI1997} for the mixed $q$-Gaussian setting, or \cite{LR2011} on the $q$-Araki--Woods algebras. The optimal time has been computed explicitly or estimated in these works. Moreover, all these semigroups mentioned above are ultracontractive \cite{BM2021}. This is related to the fact that the constant appearing in Lemma \ref{Prop2.1} is always a universal constant,  independent of $p$, for all Ornstein-Uhlenbeck semigroups mentioned above, except the quantum Ornstein-Uhlenbeck semigroups $(T_t^{(\infty)})_{t\geq 0}$ considered in this paper. Indeed, we will see that for $(T_t^{(\infty)})_{t\geq 0}$, the constant $c(\beta)p^{1/2}$ appearing in Lemma \ref{Prop2.1} does depend on $p$, which constitutes the main difficulty in estimating $t_p$. 
	This also highlights a significant difference between the quantum Ornstein-Uhlenbeck semigroups studied here and the other noncommutative Ornstein-Uhlenbeck semigroups mentioned above. In addition, we can also see that the optimal time $t_p$ for $(T_t^{(\infty)})_{t\geq 0}$ has the same order with that of the classical Ornstein-Uhlenbeck semigroups on $\real^n$.

	In particular, when $\alpha_1=-\alpha_2$ and $\alpha_3=1$, the quantum Ornstein-Uhlenbeck semigroup $T_t^{(\infty)}$ is the same as in that in \cite{CFL2000} (see \cite[Example 3.1]{KP}). On the one hand, in this special case, it has already been shown in \cite{CS2008} that $T_t^{(\infty)}$ is hypercontractive. The authors in \cite{CS2008} employed the birth and death processes in an ingenious way to show the hypercontractivity of $T_t^{(\infty)}$. In fact, they showed that $T_t^{(\infty)}$ is contractive from $L_2(\rho)$ to $L_4(\rho)$ once $t$ is larger than some fixed time. They did not consider directly the contractivity of $T_t^{(\infty)}$ from $L_2(\rho)$ to $L_p(\rho)$ for $2\leq p\neq 4<\8$. It seems that their method does not yield the best estimate of optimal time when $p$ is very large. On the other hand, the best constants of $p$-log-Sobolev inequalities of $T_t^{(\infty)}$ have been calculated explicitly in the remarkable papers by Beigi and Rahimi-Keshari \cite{BK2024} for $1\leq p\leq 2$ and by Carlen and Maas \cite{CJ2017} for $p=1$ respectively. In \cite[Corollary 5]{BK2024}, Beigi \emph{et al.} also gave an estimate of the optimal time, and in our setting, they obtained
	\[
	t_p \leq \frac{\alpha_\beta}{2\tau} \ln (p - 1), \quad \alpha_\beta = \frac{\beta}{4} \coth \sk{\frac{\beta}{4}} > 1.
	\]
	Unfortunately, their coefficient of $\ln(p-1)$ for the optimal time is not the best one.  This also reveals that in general, the best constants of log-Sobolev inequalities cannot yield the optimal time, which is quite different from the case of the classical Ornstein-Uhlenbeck semigroups. In addition, we would like to emphasize that the quantum Ornstein-Uhlenbeck semigroups which we consider in this paper are much more general than that in \cite{CFL2000}. This gives more examples of hypercontractive quantum Markov semigourps on non-tracial von Neumann algebras.
	
	The paper is organized as follows. In Section \ref{Section2}, we present more details on the eigenvalues and corresponding eigenvectors of $T_t^{(2)}$. In Section \ref{twokey}, we reduce the proof of Theorem \ref{sub-MainThm} to those of Lemma \ref{Prop2.1} and Lemma \ref{new-lemma3.4}. Section \ref{Section3} is devoted to the proof of Theorem \ref{Main-Lem}, which directly implies Lemma \ref{Prop2.1}. Finally, Theorem \ref{sub-MainThm} will be proved in Section \ref{Section5}.
	
	\bigskip
	
	\section{Preliminaries}\label{Section2}
	
	In the proof of Theorem \ref{sub-MainThm}, we need to use the eigenvalues of $\mathcal{G}^{(2)}$ and the associated eigenvectors on $S_{2}$. This section is devoted to a detailed presentation of eigenvalues and eigenvectors.
	
	From \cite[Theorem 3.2 and Theorem 5.2]{KP}, the spectra of the generator $\mathcal{G}^{(2)}$ are discrete, i.e., 
	$$ \sigma(\mathcal{G}^{(2)})=\left\{m \tau_1+n \tau_2: m, n=0,1,2, \ldots\right\},$$
	where $\tau_1=-2 \gamma \alpha_1, \tau_2=2 \gamma \alpha_2\alpha_3$. Denote by
	\begin{equation}\label{gap}
		\tau=\min\{\tau_1, \tau_2\}>0. 
	\end{equation}
	In the following, we are about to introduce eigenvectors of $\mathcal{G}^{(2)}$. 
	
	For an element $x \in S_2$, we denote its matrix entry by
	\[
	x_{mn} = \langle e_m, x(e_n) \rangle, \quad m, n \geq 0.
	\]
	Let $J: S_2 \rightarrow S_2$ be the modular conjugation defined by
	$$ J(y)= y^*, \quad \text{for} \quad y \in S_2.$$
	Let $L_y$ and $R_y$ be the left multiplication and right multiplication operators respectively. In particular, 
	\[
	L_a : {\rm dom}(L_a) \to S_2 :\  L_a(x) = ax, \quad L_{a^*} :  {\rm dom}(L_{a^*}) \to S_2 :\  L_{a^*}(x) = a^*x,  \quad x \in S_2,
	\]
	\[
	R_a : {\rm dom}(R_a) \to S_2 :\  R_a(x) = xa, \quad R_{a^*} : {\rm dom}(R_{a^*}) \to S_2 : \ R_{a^*}(x) = xa^*, \quad x \in S_2,
	\]
	where
	\[
	{\rm dom}(L_a) = {\rm dom}(L_{a^*}) = \lk{ x \in S_2 : \sum_{m, n \geq 0} m|x_{mn}|^2 < \infty }
	\]
	and
	\[
	{\rm dom}(R_a) = {\rm dom}(R_{a^*}) = \lk{ x \in S_2 : \sum_{m, n \geq 0} n|x_{mn}|^2 < \infty }.
	\] 
	Define
	$$j(L_x)=JL_xJ=R_{x^*}\in B(S_2)$$ 
	for $x\in B(H)$. 
	We write
	$$
	\begin{aligned}
		& D_1:=(2 \sinh (\beta / 2))^{-1 / 2}\left(e^{\beta / 4} L_a-e^{-\beta / 4} j\left(L_{a^*}\right)\right) \\
		& D_2:=(2 \sinh (\beta / 2))^{-1 / 2}\left(e^{-\beta / 4} L_{a^*}-e^{\beta / 4} j(L_a)\right).
	\end{aligned}
	$$
	A direct calculation shows that $D_1$ and $D_2$ satisfy the following canonical commutation relations
	$$
	\begin{gathered}
		{\left[D_i, D_j\right]=0, \quad\left[D_i^*, D_j^*\right]=0} \\
		{\left[D_i, D_j^*\right]=\delta_{i j} \mathbf{1}}.
	\end{gathered}
	$$
	Define
	$$ A_1=2^{-1 / 2}\left(D_1-D_2\right), \quad A_2=2^{-1 / 2}\left(D_1+D_2\right) $$
	on their domain
	\[
	{\rm dom}(A_1) = {\rm dom}(A_2) = \lk{ x \in S_2 : \sum_{m,n \geq 0} (m + n) |x_{mn}|^2 < \infty }.
	\]
	Then $A_1$ and $A_2$ also satisfy the canonical commutation relations. Let $\xi_0=\rho^{1/2}$. 
	From the canonical commutation relations of $A_1$ and $A_2$, one has
	$$ \|(A_1^*)^m\left(A_2^*\right)^n \xi_0\|_{2}=\sqrt{m!}\cdot \sqrt{n!}.   $$
	For $k \geq 0$, define
	\[
	L_k = {\rm span}\lk{ (a^*)^i a^j \xi_0 : 0 \leq i, j \leq k, \ i + j \leq k } \subseteq S_2.
	\]
	By induction, it is clear that
	$$  (A_1^*)^m\left(A_2^*\right)^n \xi_0\in L_{m+n}.   $$
	Denote by
	$$  \xi_{m, n}=\dfrac{1}{\sqrt{m!}\cdot\sqrt{n!}}(A_1^*)^m\left(A_2^*\right)^n \xi_0, \quad m, n=0,1,2, \cdots.  $$
	Here, $\xi_{0, 0}:=\xi_{0}$. Thus $\{\xi_{m, n}\}$ is an orthonormal basis of $S_2$; see \cite[Lemma 4.1(a)]{KP}. Let
	\[
	\mathcal{H}_{\rm fin} = {\rm span}\lk{ \xi_{m, n} : m,n \geq 0 }.
	\]
	In this paper, domains of unbounded operators acting on $S_2$ are restricted to $\mathcal{H}_{\rm fin}$ for all calculations. Moreover, in \cite[Theorem 5.2]{KP}, the authors showed that for any $m, n\in \mathbb{N}$, 
	\begin{equation}\label{Eigen}
		\mathcal{G}^{(2)}( (A_1^*)^m\left(A_2^*\right)^n \xi_0)=(m \tau_1+n \tau_2) (A_1^*)^m\left(A_2^*\right)^n \xi_0, 
	\end{equation}

	All information introduced above is already enough for the use of our proof. So we refer the interested reader to \cite{KP} for more details about the quantum Markov semigroup $(T_t^{(\infty)})_{t\geq 0}$.
	
	\section{Two key lemmas}\label{twokey}
	In this section, we aim to reduce the proof of Theorem \ref{sub-MainThm} to those of Lemma \ref{Prop2.1} and Lemma \ref{new-lemma3.4}. We would like to remark that the proof of Lemma \ref{new-lemma3.4} relies on Lemma \ref{Prop2.1}. In the sequel, we always assume $2\leq p<\8$.
	
	\subsection{First key lemma} The following is our first key lemma:
	
	\begin{lemma}\label{Prop2.1}
	For any $2<p<\8$, there exists a constant $c(\beta)>0$ such that for any $x \in L_k$, $k\in \nat$ with $\omega (x \rho^{-1/2})=0$,
		\begin{align}\label{KL}
			\left\| \rho^{1/(2p) - 1/4} x \rho^{1/(2p) - 1/4}\right\|_{p} \leq [c(\beta)p^{1/2}]^{k} \left\|  x  \right\|_{2}.   
		\end{align}
	\end{lemma}
	
		From Lemma \ref{Prop2.1}, the constant $c(\beta)p^{1/2}$ for the quantum Ornstein-Uhlenbeck semigroups depends on $p$, which will be shown in Section \ref{Section4}. However, for the other noncommutative Ornstein-Uhlenbeck semigroups mentioned before, the corresponding constant is always a universal constant.

		The proof of Lemma \ref{Prop2.1} follows a standard method which is due to Nelson \cite{EN1973}. See \cite{BZ2000,CS2008,BVY2021} for more applications of this method.

	Note that \eqref{KL} is trivial when $k=0$. Now we are going to show \eqref{KL} for $k\geq 1$. So suppose $k\geq 1$ and
	$$  x=\sum_{i+j\leq k} c_{i, j} (a^*)^i a^j \xi_0\in L_k,$$
	where $c_{i, j}\in \mathbb{C}$. In what follows, $i$ and $j$ are always non-negative integers. For $-k\leq m\leq k$, let 
	$$ x_m=\sum_{i+j\leq k, j-i=m} c_{i, j} (a^*)^i a^j \xi_0.  $$
	Denote by
	$$  d_{n, i}=\left(\dfrac{n!}{(n-i)!}\right)^{1/2}, \quad n\geq i, n\geq 0.  $$
	Then by direct calculation, for $j-i=m$,
	$$ (a^*)^i a^j=\sum_{\substack{n\geq i, n\geq 0 \\ n+m\geq j, n+m\geq 0}}^{\8} d_{n, i} d_{n+m, j} e_n\otimes e_{n+m}.  $$
	Thus,
	\begin{align*}
		x_m&=\sum_{i+j\leq k, j-i=m} c_{i, j} (a^*)^i a^j \xi_0 \\
		&=(1-e^{-\beta})^{1/2} \sum_{i+j\leq k, j-i=m}  c_{i, j} \sum_{\substack{n\geq i, n\geq 0 \\ n+m\geq j, n+m\geq 0}}^{\8} e^{-(n+m)\beta/2} d_{n, i} d_{n+m, j} e_n\otimes e_{n+m}\\
		&= (1-e^{-\beta})^{1/2} \sum_{ \substack{2i+m\leq k \\  i\geq 0, i+m\geq 0} }  c_{i, i+m} \sum_{ n\geq i} e^{-(n+m)\beta/2} d_{n, i} d_{n+m, i+m} e_n\otimes e_{n+m}\\
		&= (1-e^{-\beta})^{1/2} \sum_{n=0}^\8   \left( \sum_{ \substack{2i+m\leq k \\ i+m\geq 0, 0\leq i \leq n} } c_{i, i+m}e^{-(n+m)\beta/2} d_{n, i} d_{n+m, i+m} \right) e_n\otimes e_{n+m}.
	\end{align*}
	Denote by
	$$ f_{k, n, m}= \sum_{ \substack{2i+m\leq k \\ i+m\geq 0, 0\leq i \leq n} } c_{i, i+m}d_{n, i} d_{n+m, i+m}.   $$
	For convenience, we set $f_{k, n, m} = 0$ whenever $n + m < 0$. Direct calculation shows that
	\begin{align*}
		& \rho^{1/(2p)-1/4} x_m \rho^{1/(2p)-1/4} \\
		=\ & (1-e^{-\beta})^{1/p}  \sum_{n=0}^\8   \left( e^{-n\beta(1/(2p)-1/4)} e^{-(n+m)\beta(1/(2p)-1/4)} e^{-(n+m)\beta/2} f_{k, n, m}  \right) e_n\otimes e_{n+m} \\
		=\ & (1-e^{-\beta})^{1/p} e^{-m\beta(1/(2p) + 1/4)} \sum_{n=0}^\8 e^{-n\beta/p}  f_{k, n, m} e_n\otimes e_{n+m} .   
	\end{align*}
	Then for all $p \geq 2$,
	\begin{align}\label{eq:p-norm}
		& \left\| \rho^{1/(2p)-1/4} x_m \rho^{1/(2p)-1/4}\right\|_{p} \notag\\
		=\ & (1-e^{-\beta})^{1/p} e^{-m\beta(1/(2p) + 1/4)} \sk{ \sum_{n \geq 0} e^{-n \beta} |f_{k, n, m}|^p }^{1/p}.
	\end{align}
	In particular,
	\begin{align*}
		\| x_m \|_2 & = (1-e^{-\beta})^{1/2} e^{-m\beta/2} \sk{ \sum_{n \geq 0} e^{-n \beta} |f_{k, n, m}|^2 }^{1/2}.
	\end{align*}
	
	With these notations above, we have the following proposition.
	\begin{proposition}\label{Prop2.2}
		The inequality \eqref{KL} holds if there exists a constant $C$ depending only on $\beta$ such that for all $k \geq 1$ and $-k \leq m \leq k$,
		\begin{equation}\label{KL3}
			\left( \sum_{n=0}^\8  e^{-n\beta}| f_{k, n, m} |^p  \right)^{1/p}  \leq e^{m\beta(1/2p-1/4)} [Cp^{1/2}]^k \left( \sum_{n=0}^\8 e^{-n\beta} | f_{k, n, m} |^2  \right)^{1/2}. 
		\end{equation}
		Moreover, the constant $c(\beta)$ in \eqref{KL} satisfies 
		\[
		c(\beta) \leq \sqrt{3}(1 - e^{-\beta})^{-1/2} C.
		\]
	\end{proposition} 
	
	\begin{proof}
		We decompose $x$ as the sum of $x_m$, and this idea comes from \cite{CS2008}. At first, we estimate $x_m$. Applying \eqref{KL3} to \eqref{eq:p-norm}, one has
		\begin{align*}
			\left\| \rho^{1/(2p)-1/4} x_m \rho^{1/(2p)-1/4}\right\|_{p} & \leq [Cp^{1/2}]^k(1-e^{-\beta})^{1/p} e^{-m\beta/2} \sk{ \sum_{n \geq 0} e^{-n \beta} |f_{k, n, m}|^2 }^{1/2} \\
			& = [Cp^{1/2}]^k(1-e^{-\beta})^{1/p - 1/2} \| x_m \|_2 \leq \mk{ C (1 - e^{-\beta})^{-1/2}p^{1/2} }^k \| x_m \|_2. 
		\end{align*}
		Then \eqref{KL} follows from
		\begin{align*}
			\left\| \rho^{1/(2p)-1/4} x \rho^{1/(2p)-1/4}\right\|_{p} & \leq \sum_{m=-k}^{k}\left\| \rho^{1/(2p)-1/4} x_m \rho^{1/(2p)-1/4}\right\|_{p}\\
			& \leq \sum_{m=-k}^{k} \mk{ C (1 - e^{-\beta})^{-1/2}p^{1/2} }^k \left\| x_m \right\|_{2} \\
			& \leq \mk{ C (1 - e^{-\beta})^{-1/2}p^{1/2} }^k (2k+1)^{1/2} \|x\|_2 \\
			& \leq \mk{\sqrt{3} (1 - e^{-\beta})^{-1/2}C p^{1/2} }^k \| x \|_2.
		\end{align*}
		This completes the proof.
	\end{proof} 
	
	Our goal is to show the following theorem, which directly implies \eqref{KL3}.
	
	\begin{thm}\label{Main-Lem}
		Let $f_{k, n, m}$ be defined as before. For all $k \geq 1$, $-k \leq m \leq k$, there exists a constant $C(\beta)$ such that
		\begin{equation}
			\left( \sum_{n=0}^\8  e^{-n\beta}| f_{k, n, m} |^p  \right)^{1/p}  \leq e^{m\beta(1/2p-1/4)} \mk{ C(\beta)p }^{ k/2 } \left( \sum_{n=0}^\8 e^{-n\beta} | f_{k, n, m} |^2  \right)^{1/2}.
		\end{equation}
	\end{thm}
	
	The proof of Theorem \ref{Main-Lem} is quite technical, and so we put the detailed proof to Section \ref{Section3}.
	\subsection{Second key lemma} The following lemma concerns the estimate of mixed moments, which follows from Lemma \ref{Prop2.1}.
	
	\begin{lemma}\label{new-lemma3.4}
		Fix $n\in \mathbb{N}\setminus \{0\}$. For $1\leq k\leq n$, let $y_k\in L_k$ with
		$\omega(y_k\rho^{-1/2})=0$. Assume that there exist $\eta\geq0$ and $b_k\geq0$ such that for all $1\leq k\leq n$,
		$$ \norm{y_k}_{2}\leq\eta^k b_k.$$ 
		Therefore, there exists $d(\beta)>0$ depending only on $\beta$ such that for every $c\in\mathbb C$ and every $2<p<\infty$, if $\eta\leq d(\beta)p^{-1/2}$, then 
		\begin{equation}\label{eq:mixed-contraction}
			\norm{\rho^{1/(2p) - 1/4} \left(c \cdot \rho^{1/2}+\sum_{k=1}^n y_k\right) \rho^{1/(2p) - 1/4} }_{p}
			\leq\left(|c|^2+\sum_{k=1}^n b_k^2\right)^{1/2}.
		\end{equation}
		\end{lemma}
		
		\begin{proof}
			For $h\in [-1, 1]$, define the modular conjugation
			$$ M_h(y)=\rho^h y\rho^{-h}. $$
			Then $M_h(L_k)\subset L_k$. For any matrix $(z_{rs})$, it multiplies each entry $z_{rs}$ by $e^{-\beta h(r-s)}$.
			As $|r-s|\leq k$ on the support of $y\in L_k$,
			\begin{equation}\label{eq:modular-bound}
				\norm{M_h(y)}_{2}\leq e^{\beta k}\norm{y}_{2}.
			\end{equation}
			By Lemma \ref{Prop2.1}, one has
			\begin{align}\label{new-3.9}
				\left\| \rho^{1/(2p) - 1/4} y_k \rho^{1/(2p) - 1/4}\right\|_{p} \leq [c(\beta)p^{1/2}]^k \left\|  y_k  \right\|_{2}.   
			\end{align}
			For $2\leq j\leq p$, \eqref{eq:modular-bound} and \eqref{new-3.9} yield
			$$  	\left\| \rho^{1/(2j) - 1/4} M_h(y_k) \rho^{1/(2j) - 1/4}\right\|_{j} \leq [c(\beta)j^{1/2}]^k \left\|  M_h(y_k)  \right\|_{2}\leq (A_j\eta)^k b_k,     $$
			where $A_j=e^\beta c(\beta)j^{1/2} $.
			The same estimate holds for $y_k^*$, since $M_h(y_k^*)=M_{-h}(y_k)^*$. Let $b=\left(\sum\limits_{k=1}^n b_k^2\right)^{1/2}$, and
			$y=\sum\limits_{k=1}^n y_k$. Note that by Lemma \ref{Prop2.1}, $A_j = \sqrt{j/p} A_p$.

			At first, we consider the case where $p$ is even. Assume that $A_p\eta\leq1/16$, i.e. 
			$$  \eta \leq \dfrac{1}{16} A_p^{-1} = \frac{e^{-\beta}}{16 c(\beta)} p^{-1/2}.  $$
		 Let $2\leq j\leq p$. Then for $z=y$ or $y^*$, one has
			\begin{align}
			\left\| \rho^{1/(2j) - 1/4} M_h(z) \rho^{1/(2j) - 1/4}\right\|_{j}&\leq \sum_{k=1}^n (A_j\eta)^k b_k \leq b \left(\sum_{k=1}^n (A_j\eta)^{2k}\right)^{1/2}\notag \\
			&\leq \dfrac{bA_j \eta}{\sqrt{1-(A_j \eta)^2}}\leq 2bA_j \eta. \label{new-3.10}
			\end{align}
			If $b=0$ or $\eta=0$, the desired estimate follows directly. As $p$ is even, we can expand the left-hand-side of \eqref{eq:mixed-contraction}. To this end, we need to estimate the following
			$$ \left|\Tr(z_1\rho^{d_1}\cdots z_j\rho^{d_j})\right|,  $$
			where each $z_i$ is $\rho^{-1/4}y\rho^{-1/4}$ or $\rho^{-1/4}y^*\rho^{-1/4}$, and $d_i\geq0$ with $\sum\limits_{i=1}^j d_i=1$. Now put
			\[
			h_1=0,\qquad h_{i+1}=h_i+d_i-\frac1j\quad(1\leq i<j).
			\]
			Then $|h_i|\leq1$, and the same recursion gives $h_{j+1}=h_1$. Let $W_i=\rho^{1/(2j)}M_{h_i}(z_i)\rho^{1/(2j)}$,
			the exponent of $\rho$ between $z_i$ and $z_{i+1}$ in
			$\Tr(W_1\cdots W_j)$ is
			$1/j-h_i+h_{i+1}=d_i$. Therefore, by the tracial property of $\text{Tr}$ and by \eqref{new-3.10}, one has
			\begin{align}
				\left|\Tr(z_1\rho^{d_1}\cdots z_j\rho^{d_j})\right|
				&=|\Tr(W_1\cdots W_j)|\notag\\
				&\leq\prod_{i=1}^j\norm{W_i}_{j}
				\leq(2bA_j \eta)^j, \label{eq:word-bound}
			\end{align}
			where we have applied $W_i=\rho^{1/(2j) - 1/4} M_{h_i}(z) \rho^{1/(2j) - 1/4}$ for $z=y$ or $y^*$.
			
			Now we can expand the left-hand-side of \eqref{eq:mixed-contraction}. Write $p=2m$,
			$$ Z_p =c \cdot \rho^{1/p} + \rho^{1/(2p)} (\rho^{-1/4} y \rho^{-1/4}) \rho^{1/(2p)}. $$ 
			Expand $\Tr((Z_p^*Z_p)^m)$ in its $p$ alternating adjoint and nonadjoint items. The term with no occurrence of $y$ or $y^*$ is $|c|^p$.
			Every term with one such occurrence is zero, because it is a scalar multiple of $\omega(y\rho^{-1/2})$ or $\omega(y^*\rho^{-1/2})$.
			For a term with $j\geq2$ such occurrences, the intervening density powers combine into nonnegative gaps $d_i$ with sum one. There are $\binom pj$ choices of items, and the remaining scalar multiplier has modulus $|c|^{p-j}$.
			Thus \eqref{eq:word-bound} gives
			$$ \| Z_p \|_p^p=\Tr((Z_p^*Z_p)^m)\leq |c|^p+\sum_{j=2}^p C_j |c|^{p-j}b^j,  $$
		where
		$$ C_j=\binom pj D^j\left(\frac jp\right)^{j/2},
		\qquad D=2A_p\eta \leq\frac18.  $$
			Now we estimate $C_{j}$. At first, we consider $j=2\ell$ for $1\leq\ell\leq m$. Note that
			\[
			\frac{\binom{2m}{2\ell}}{\binom m\ell}
			=\frac{\ell!}{(2\ell)!}
			\prod_{i=0}^{\ell-1}2(2m-2i-1)
			\leq\frac{(2p)^\ell}{\ell^\ell},
			\]
			since $(2\ell)!/\ell!\geq\ell^\ell$.
			Consequently,
			\begin{equation}\label{eq:even-coefficients}
				C_{2\ell}\leq\binom m\ell(4D^2)^\ell.
			\end{equation}
			For odd $3\leq j\leq p-1$, the convexity of $u\log u$ implies
			$$  j^{j/2}\leq \sqrt{(j-1)^{(j-1)/2} (j+1)^{(j+1)/2}}  $$
			and
			\[
			\frac{\binom pj^2}{\binom p{j-1}\binom p{j+1}}
			=\frac{j+1}{j}\frac{p-j+1}{p-j}\leq4
			\]
			implies $C_j\leq2\sqrt{C_{j-1}C_{j+1}}$. Hence, one has
			\[
			C_j |c|^{p-j}b^j
			\leq C_{j-1}|c|^{p-j+1}b^{j-1}
			+C_{j+1}|c|^{p-j-1}b^{j+1}.
			\]
			Each positive even index contributes to at most two of these bounds besides its original term. Therefore
			\[
			\sum_{j=2}^p C_j |c|^{p-j}b^j
			\leq3\sum_{\ell=1}^m C_{2\ell}|c|^{p-2\ell}b^{2\ell}.
			\]
			As $4D^2\leq1/16$, we have $3(4D^2)^\ell\leq1$ for all $\ell\geq1$. By \eqref{eq:even-coefficients}, we obtain
			\[
			\| Z_p \|_p^p
			\leq |c|^p+\sum_{\ell=1}^m\binom m\ell
			|c|^{p-2\ell}b^{2\ell}
			=(|c|^2+b^2)^m.
			\]
			This proves \eqref{eq:mixed-contraction} for even $p$, including $p=2$, when there are no odd terms.
			
			For any $p\geq2$, take $n\in \mathbb{Z}$, so that
			$2n-2<p\leq 2n$. Let
			\[
			d(\beta) = \frac{e^{-\beta}}{16 c(\beta)}.
			\]
			If $\eta \leq d(\beta)p^{-1/2} $, then
			$$ A_{2n} \eta \leq \sqrt{2} A_p \eta \leq \dfrac{\sqrt{2}}{16}, $$
			and in this case, for all $\ell\geq1$,
			$$ 3(4 (2A_{2n} \eta)^2)^\ell\leq1. $$ 
			Denote $r = 1/(2p) - 1/(4n)$. Apply the result just proved for $2n$ and then we have
			$$ \| Z_p \|_p = \| \rho^r Z_{2n} \rho^r \|_p \leq \| \rho^r \|_{1/r}^2 \| Z_{2n} \|_{2n}\leq (|c|^2+b^2)^{1/2}. $$
		This finishes the proof.
				\end{proof}

	\subsection{One important inqeuality}
	We introcude Lemma \ref{new-lemma3.5}, which is needed in the proof of Theorem \ref{Main-Lem}. Let $\Gamma(s)$ be the Gamma function. We will need the following Stirling's formula, see \cite{A1965}. For all $s \geq 1$, one has
	\begin{equation}\label{Stirling}
		\sqrt{2\pi s}\sk{ \frac{s}{e} }^s \leq \Gamma(1 + s) \leq e\sqrt{s} \sk{ \frac{s}{e} }^s.
	\end{equation}
	
	\begin{lemma}\label{new-lemma3.5}
		For all $s\geq 1$,
		\begin{equation}\label{Stirling-2}
			e^{- \beta} \beta^{-(s + 1)} \sqrt{ \frac{ \pi s}{2} } \sk{ \frac{s}{e} }^s \leq \sum_{n \geq 0} e^{-n\beta} n^s \leq e(1 + \beta) \beta^{-(s+1)}\sqrt{s} \sk{ \frac{s}{e} }^s,
		\end{equation}
	\end{lemma}
	\begin{proof}
		Define for $x\geq 0$
		\[
		h(x) = e^{-\beta x} x^s, \quad \beta > 0, \ s \geq 1.
		\]
		It is clear that $h'(x) = e^{-\beta x} x^{s-1}( -\beta x + s )$. Thus $h(x)$ increases in $[0, s/\beta]$, decreases in $[s/\beta, \infty)$ and achieves its maximum at $x = s/\beta$, i.e.
		\[
		h(x) \leq h\sk{ \frac{s}{\beta} } = \beta^{-s} \sk{ \frac{s}{e} }^s.
		\]
		Assume that there exists a nonnegative integer $m$ such that $ m \leq s/\beta < m + 1$. 
		
		At first, we show the right-sided inequality. If $m \geq 1$, then
		\begin{align*}
			\sum_{n \geq 0} e^{-n\beta} n^s & = \sum_{n = 0}^{m-1} h(n) +\sum_{n = m}^\infty h(n) \leq \sum_{n = 0}^{m-1} \int_{n}^{n+1} h(x)\, \d x + h\sk{\frac{s}{\beta}} + \sum_{n = m  +1}^\infty \int_{n-1}^{n} h(x)\, \d x \\
			& = \int_0^\infty e^{-\beta x}x^s \,\d x + h\sk{ \frac{s}{\beta} } = \beta^{-(s + 1)} \Gamma(1 + s) + \beta^{-s} \sk{ \frac{s}{e} }^s \\
			& \leq e\sqrt{s} \beta^{-(s + 1)} \sk{ \frac{s}{e} }^s + \beta^{-s} \sk{ \frac{s}{e} }^s \leq  e\beta^{-(s + 1)}(1 + \beta) \sqrt{s} \sk{ \frac{s}{e} }^s,
		\end{align*}
		where we use $s \geq 1$ in the last inequality. The case $m = 0$ follows similarly.
		
		For the remaining inequality, we consider two cases. First, suppose $ s > 2\beta^2/\pi $. If $m\geq 1$, then
		\begin{align*}
			\sum_{n \geq 0} e^{-n\beta} n^s & = \sum_{n = 1}^{m} h(n) + \sum_{n = m+1}^\infty h(n) \geq \sum_{n = 1}^{m} \int_{n-1}^{n} h(x)\, \d x + \sum_{n = m+1}^\infty \int_{n}^{n+1} h(x)\, \d x \\
			& = \int_0^\infty h(x)\,\d x - \int_{m}^{m+1} h(x)\, \d x \geq \int_0^\infty e^{-\beta x}x^s \,\d x - h\sk{ \frac{s}{\beta} } \\
			& = \beta^{-(s + 1)} \Gamma(1 + s) - \beta^{-s} \sk{ \frac{s}{e} }^s > \sqrt{2 \pi s} \beta^{-(s + 1)} \sk{ \frac{s}{e} }^s - \beta^{-s} \sk{ \frac{s}{e} }^s \\
			& = \beta^{-(s + 1)}(\sqrt{2\pi s} - \beta) \sk{ \frac{s}{e} }^s > \beta^{-(s + 1)}\sk{ \sqrt{2\pi s} - \sqrt{ \frac{\pi s}{2} } } \sk{ \frac{s}{e} }^s \\
			& = \beta^{-(s+1)} \sqrt{ \frac{ \pi s}{2} } \sk{ \frac{s}{e} }^s.
		\end{align*}
		The case $m = 0$ follows similarly. 
		
		Now for $ 1 \leq s \leq 2\beta^2/\pi $, recall that $ m \leq s/\beta < m + 1$, thus
		\begin{align*}
			\sum_{n \geq 0} e^{-n\beta} n^s & > e^{-(m + 1)\beta} (m + 1)^s > e^{-(s + \beta)}\sk{ \frac{s}{\beta} }^s = e^{-\beta}\beta^{-s} \sk{ \frac{s}{e} }^s \\
			& = e^{-\beta} \beta^{-(s + 1)} \frac{\beta}{\sqrt{s}} \sqrt{s}\sk{ \frac{s}{e} }^s \geq e^{- \beta} \beta^{-(s + 1)} \sqrt{ \frac{ \pi s}{2} } \sk{ \frac{s}{e} }^s.
		\end{align*}
		This proves \eqref{Stirling-2}.
	\end{proof}
	\bigskip
	
	\section{Proof of Theorem \ref{Main-Lem}}\label{Section3}
	In this section, we will show Theorem \ref{Main-Lem}, which finishes the proof of Lemma \ref{Prop2.1}. To begin with, we introduce Meixner polynomials, which are the main ingredient for our proof of Theorem \ref{Main-Lem}. Recall that $p\geq 2$.
	\subsection{Orthogonal polynomials}
	We present a special family of Meixner polynomials, which is orthogonal with respect to the weight $w(n) = e^{-n\beta}$ ($n\in \nat$). For any $k \geq 0$, define
	\begin{equation}
		M_k(n) = e^{-k\beta}\sum_{j = 0}^k (-1)^{j} (e^{\beta} - 1)^j \binom{k}{j} \binom{n}{j}, \quad \forall \, n\in\nat.
	\end{equation}
	We have the following proposition, see \cite[Chapter 9]{koekoek2010}.
	\begin{proposition}\label{Orthogonal}
		The family $\lk{ M_0, M_1, \cdots, M_n, \cdots }$ is orthogonal with respect to the weight $w(n) = e^{-n \beta}$. More precisely, for any $m, l\in \nat$
		\begin{equation}\label{Orthogonal-eq1}
			\sum_{n = 0}^\infty e^{-n \beta} M_m(n) M_l(n) = \frac{ e^{-m\beta} }{ 1 - e^{-\beta} } \delta_{ml}, 
		\end{equation}  
		where $\delta_{ml}$ is the Kronecker delta function. 
	\end{proposition}
	
	The following lemma is needed later.
	\begin{lemma}
		For any $k \in \nat$, we have
		\begin{equation}\label{Orthogonal-eq2}
			|M_k(n)| \leq \max\lk{ \frac{n^k}{k!}, \frac{k^k}{k!} }, \quad \forall\, n \in \nat.
		\end{equation}
	\end{lemma}
	
	\begin{proof}
		We observe that when $j \leq n$ and $j\leq k$,
		\[
		\binom{n}{j} = \frac{ n(n - 1) \cdots (n - j + 1) }{j!} \leq \frac{n^j}{j!} = \frac{ n^j (j + 1)\cdots k }{k!} \leq \frac{n^j k^{k - j}}{k!}.
		\]
		We consider two cases. For $0 \leq n \leq k$, since $\binom{n}{j} = 0$ if $j > n$, we have
		\begin{align*}
			|M_k(n)| & \leq e^{-k\beta}\sum_{j = 0}^n (e^{\beta} - 1)^j \binom{k}{j} \binom{n}{j} \\
			& \leq e^{-k\beta}\sum_{j = 0}^n (e^{\beta} - 1)^j \binom{k}{j} \frac{n^j k^{k - j}}{k!} \\
			& \leq e^{-k\beta}\sum_{j = 0}^k (e^{\beta} - 1)^j \binom{k}{j} \frac{k^k}{ k! } = \frac{k^k}{k!}.
		\end{align*}
		For $n > k$, we have
		\begin{align*}
			|M_k(n)| & \leq e^{-k\beta}\sum_{j = 0}^k (e^{\beta} - 1)^j \binom{k}{j} \frac{n^j k^{k - j}}{k!} \\
			& \leq e^{-k\beta}\sum_{j = 0}^k (e^{\beta} - 1)^j \binom{k}{j} \frac{n^k}{ k! } = \frac{n^k}{k!}.
		\end{align*}
		This completes the proof.
	\end{proof}

	\subsection{Proof of Theorem \ref{Main-Lem}}
	We divide our proof of Theorem \ref{Main-Lem} into two steps. First, we prove that Theorem \ref{Main-Lem} holds for all $k \geq 1$, and $m = 0$ or $m=1$. Note that $m = 0$ implies
	\[
	\sum_{ \substack{2i+m\leq k \\ i+m\geq 0, 0\leq i \leq n} } c_{i, i+m}d_{n, i} d_{n+m, i+m} = \sum_{ \substack{2i \leq k \\ 0\leq i \leq n} } c_{i, i}d_{n, i}^2 = \sum_{ \substack{2i \leq k \\ 0\leq i \leq n } } c_{i, i} \dfrac{n!}{(n-i)!}.
	\]
	Hence $f_{k, n, 0}$ is a polynomial with respect to $n$ with degree not larger than $k/2$. We will prove Theorem \ref{Main-Lem} for all $k \geq 1$ and $m = 0$ by showing the following slightly stronger proposition below.
	\begin{lemma}\label{Main-m=0}
		For all $k \in \nat$ and any polynomial $p_k$ with degree $\deg p_k = k$, there exists a constant $C_1(\beta)$ such that
		\[
		\left( \sum_{n=0}^\8  e^{-n\beta}| p_k(n) |^p  \right)^{1/p}  \leq \mk{ C_1(\beta)p }^{k} \left( \sum_{n=0}^\8 e^{-n\beta} | p_k(n) |^2  \right)^{1/2}. 
		\]
	\end{lemma}
	
	\begin{proof}
		The case $k = 0$ is trivial. For $k \geq 1$, by Proposition \ref{Orthogonal}, there exist $c_i \in \com$ $(i=0, 1, \cdots , k)$ such that
		\[
		p_k(n) = \sum_{i = 0}^k c_i M_i(n), \quad \forall \, n\in\nat.
		\]
		Thus \eqref{Orthogonal-eq1} implies
		\begin{align}\label{L_k-l2}
			\left( \sum_{n=0}^\8 e^{-n\beta} | p_k(n) |^2  \right)^{1/2} & = \sk{ \sum_{i = 0}^k \sum_{n=0}^\8 |c_i|^2 e^{-n\beta} | M_i(n) |^2  }^{1/2} \notag\\
			& = \sk{ \sum_{i = 0}^k |c_i|^2 \frac{e^{-i \beta}}{ 1 - e^{-\beta} } }^{1/2} .
		\end{align}
		By the triangle inequality, the left-hand-side of Lemma \ref{Main-m=0} reads
		\begin{align*}
			\left( \sum_{n=0}^\8  e^{-n\beta}| p_k(n) |^p  \right)^{1/p} & \leq \sum_{i = 0}^k |c_i| \left( \sum_{n=0}^\8  e^{-n\beta}| M_i(n) |^p  \right)^{1/p}.
		\end{align*}
		Applying \eqref{Orthogonal-eq2}, \eqref{Stirling} and \eqref{Stirling-2}, one has
		\begin{align}\label{Orthogonal-eq3}
			\sum_{n=0}^\8  e^{-n\beta}| M_i(n) |^p & = \sum_{n=0}^i  e^{-n\beta}| M_i(n) |^p + \sum_{n=i + 1}^\8  e^{-n\beta}| M_i(n) |^p \notag \\
			& \leq \sum_{n=0}^i  e^{-n\beta} \frac{i^{ip} }{ (i!)^p} + \sum_{n=i + 1}^\8  e^{-n\beta} \frac{n^{ip} }{ (i!)^p}
			\notag \\
			& \leq \frac{ 1 }{ 1 - e^{-\beta} }  \frac{i^{ip} }{ \sk{2\pi i}^{p/2} \sk{ e^{-1}i }^{ip} } + \frac{ e(1 + \beta)\beta^{-ip - 1} \sqrt{ip} \sk{ e^{-1}ip }^{ip} }{ \sk{2\pi i}^{p/2} \sk{ e^{-1}i }^{ip} } \notag \\
			& \leq \sk{2\pi i}^{-p/2}\sqrt{ip} \sk{ \frac{1}{1 - e^{-\beta}} + e\sk{ 1 + \frac{1}{\beta} } } \sk{ e + \beta^{-1} }^{ip} p^{ip} \notag \\
			& \leq C_2(\beta)^{kp}p^{kp},
		\end{align}
		where we can take 
		\[
		C_2(\beta) = \sk{ \frac{1}{1 - e^{-\beta}} + e\sk{ 1 + \frac{1}{\beta} } } \sk{ e + \beta^{-1} }.
		\]
		Hence by \eqref{L_k-l2},
		\begin{align*}
			\left( \sum_{n=0}^\8  e^{-n\beta}| p_k(n) |^p  \right)^{1/p} & \leq \sum_{i = 0}^k |c_i| C_2(\beta)^k p^{k} \\
			& \leq C_2(\beta)^{k} p^{k}\sk{ \sum_{i = 0}^k |c_i|^2 \frac{e^{-i\beta}}{1 - e^{-\beta}} }^{1/2} \sk{ \sum_{i = 0}^k e^{i\beta}(1 - e^{-\beta}) }^{1/2} \\
			& \leq C_2(\beta)^{k} p^{k} e^{k\beta/2} \left( \sum_{n=0}^\8 e^{-n\beta} | p_k(n) |^2  \right)^{1/2}\\
			& \leq \mk{C_1(\beta)p}^k \left( \sum_{n=0}^\8 e^{-n\beta} | p_k(n) |^2  \right)^{1/2},
		\end{align*}
		where
		\[
		C_1(\beta) := e^{\beta/2} C_2(\beta).
		\]
		This completes the proof.
	\end{proof}
	
	\begin{rem}
		In Lemma \ref{Main-m=0}, the order $p^k$ is optimal if we take $p_k(n)=n^k$ for any $n\in\mathbb{N}$.
		\end{rem}
		
	\begin{rem}
		In particular, when $p=4$, Lemma \ref{Main-m=0} has already been shown in \cite[Proposition 3.2]{CS2008} by virtue of the birth and death processes, which is the key inequality in \cite{CS2008}. So Lemma \ref{Main-m=0} is new for all $2\leq p< \infty$ except $p = 4$. Our proof of Lemma \ref{Main-m=0} is more direct, and also applies to all $2\leq p< \infty$. In addition, we also obtain the optimal order $p^k$ for the inequality in Lemma \ref{Main-m=0}.
		\end{rem}
	
	\begin{rem}
		One natural and naive idea is to apply Lagrange interpolation formula to show Lemma \ref{Main-m=0}. However, we have once tried this naive idea before the completion of our paper, and it seems that such idea does not work effectively.
		\end{rem}
	
	We also need the following important lemma.
	\begin{lemma}\label{Main-m=1}
		For all $k \in \nat$ and any polynomial $p_k$ with degree $\deg p_k = k$, there exists a constant $C_3(\beta)$ such that
		\[
		\left( \sum_{n=0}^\8  e^{-n\beta}| \sqrt{n + 1} p_k(n) |^p  \right)^{1/p}  \leq \mk{ C_3(\beta)p }^{(2k + 1)/2} \left( \sum_{n=0}^\8 e^{-n\beta} | \sqrt{n + 1} p_k(n) |^2  \right)^{1/2}. 
		\]
	\end{lemma}
	
	\begin{proof}
		For $k = 0$, this inequality follows from Lemma \ref{new-lemma3.5} by comparing two sides. For $k \geq 1$, we consider two cases. For $p \geq 4$, applying Lemma \ref{Main-m=0} to the polynomial $(n + 1)|p_k(n)|^2$, one has
		\begin{align}\label{Orthogonal-Main-eq}
			\left( \sum_{n=0}^\8  e^{-n\beta}| \sqrt{n + 1} p_k(n) |^p  \right)^{1/p} & = \left( \sum_{n=0}^\8  e^{-n\beta}\sk{ (n + 1) |p_k(n)|^2 }^{p/2}  \right)^{1/p} \notag\\
			& \leq \sk{ \mk{C_1(\beta)p}^{2k+1} \left( \sum_{n=0}^\8 e^{-n\beta} \sk{  (n + 1)|p_k(n)|^2 }^2  \right)^{1/2} }^{1/2} \notag\\
			& = \mk{C_1(\beta)p}^{(2k+1)/2}  \sk{ \sum_{n=0}^\8 e^{-n\beta} (n + 1)^2|p_k(n)|^4  }^{1/4}.
		\end{align}
		Note that $(n + 1)^2 \leq (n + 1)^4$ for all $n \in \nat$, and applying Lemma \ref{Main-m=0} again to the polynomial $(n + 1)p_k(n)$ for $p = 4$, we obtain
		\begin{align}\label{Orthogonal-eq4}
			\sk{ \sum_{n=0}^\8 e^{-n\beta} (n + 1)^2|p_k(n)|^4  }^{1/4} & \leq \sk{ \sum_{n=0}^\8 e^{-n\beta} |(n + 1)p_k(n)|^4  }^{1/4} \notag \\
			& \leq \mk{4C_1(\beta)}^{k + 1}\sk{ \sum_{n=0}^\8 e^{-n\beta} |(n + 1)p_k(n)|^2 }^{1/2} .
		\end{align}
		By Proposition \ref{Orthogonal}, there exist $c_i \in \com$ $(i=0, 1, \cdots , k)$ such that
		\[
		p_k(n) = \sum_{i = 0}^{k} c_i M_i(n).
		\]
		Applying \eqref{Orthogonal-eq2}, we obtain
		\begin{align*}
			|np_k(n)| & \leq \sum_{i = 0}^{k} |c_i| n |M_i(n)| \leq \sum_{i = 0}^{k} |c_i| n \max\lk{ \frac{k^k}{k!}, \frac{n^k}{k!} } \leq  \sum_{i = 0}^{k} |c_i| \max\lk{ \frac{k^{k + 1}}{k!}, \frac{n^{k  +1}}{k!} } \\
			& \leq \max\lk{ \frac{k^{k + 1}}{k!}, \frac{n^{k  +1}}{k!} } \sk{ \sum_{i = 0}^k |c_i|^2 \frac{ e^{-i\beta} }{ 1 - e^{-\beta} } }^{1/2} \sk{ \sum_{i = 0}^k e^{i\beta} \sk{ 1 - e^{-\beta} } }^{1/2} \\
			& \leq e^{k\beta} \max\lk{ \frac{k^{k + 1}}{k!}, \frac{n^{k  +1}}{k!} } \sk{ \sum_{n=0}^\8 e^{-n\beta} |p_k(n)|^2 }^{1/2}.
		\end{align*}
		Hence by \eqref{Stirling} and \eqref{Stirling-2},
		\begin{align*}
			\sum_{n=0}^\8 e^{-n\beta} |np_k(n)|^2 & \leq e^{2k\beta}\sk{ \sum_{n=0}^\8 e^{-n\beta} |p_k(n)|^2 } \sk{ \sum_{n = 0}^k e^{-n \beta} \left(\frac{k^{k  +1}}{k!} \right)^2+ \sum_{n = k + 1}^\8 e^{-n \beta} \left(\frac{n^{k + 1}}{k!} \right)^2 } \\
			& \leq e^{2k\beta}\sk{ \sum_{n=0}^\8 e^{-n\beta} |p_k(n)|^2 } \\
			& \qquad \cdot \sk{ \frac{ e^{2k} k}{ 2\pi (1 - e^{-\beta}) } + \frac{ 4^{k+1} (1 + \beta)\beta^{-2k - 3}(k + 1) \sqrt{2k  +2} }{ 2\pi e }\sk{ 1 + \frac1k }^{2k + 1} } \\
			& \leq C_4(\beta)^k \sum_{n=0}^\8 e^{-n\beta} |p_k(n)|^2,
		\end{align*}
		where we can take
		\[
		C_4(\beta) = \frac{ 2e^{2\beta+2} }{ 1 - e^{-\beta} } + 16e^{2\beta+2}( \beta^{-3} + \beta^{-2} )  \max\lk{ 1, \beta^{-2} }.
		\]
		Then we obtain
		\begin{align*}
			\sum_{n=0}^\8 e^{-n\beta} |(n + 1)p_k(n)|^2 & = \sum_{n=0}^\8 e^{-n\beta} |np_k(n)|^2  + 2 \sum_{n=0}^\8 e^{-n\beta} n|p_k(n)|^2 + \sum_{n=0}^\8 e^{-n\beta} |p_k(n)|^2 \\
			& \leq \sk{C_4(\beta)^k + 1} \sum_{n=0}^\8 e^{-n\beta} |p_k(n)|^2 + 2 \sum_{n=0}^\8 e^{-n\beta} n|p_k(n)|^2 \\
			& \leq \sk{C_4(\beta)^k + 3}\sum_{n=0}^\8 e^{-n\beta} (n + 1)|p_k(n)|^2.
		\end{align*}
		Plug this into \eqref{Orthogonal-eq4}, and we obtain
		\begin{align*}
			\sk{ \sum_{n=0}^\8 e^{-n\beta} (n + 1)^2|p_k(n)|^4  }^{1/4} & \leq \mk{4C_1(\beta)}^{k + 1} \sk{C_4(\beta)^k + 3}^{1/2} \sk{ \sum_{n=0}^\8 e^{-n\beta} (n + 1)|p_k(n)|^2 }^{1/2} \\
			& \leq C_5(\beta)^{(2k + 1)/2}\sk{ \sum_{n=0}^\8 e^{-n\beta} (n + 1)|p_k(n)|^2 }^{1/2},
		\end{align*}
		where 
		\[
		C_5(\beta) = [4C_1(\beta)]^2 \sk{C_4(\beta) + 3}.
		\]
		Substituting this into \eqref{Orthogonal-Main-eq} yields
		\begin{align*}
			\left( \sum_{n=0}^\8  e^{-n\beta}| \sqrt{n + 1} p_k(n) |^p  \right)^{1/p} \leq \mk{C_1(\beta)C_5(\beta)p}^{(2k + 1)/2}\sk{ \sum_{n=0}^\8 e^{-n\beta} (n + 1)|p_k(n)|^2 }^{1/2}
		\end{align*}
		This completes the proof for $p \geq 4$.
		
		For $2 \leq p \leq 4$, by the H\"{o}lder inequality, we have
		\begin{align*}
			\left( \sum_{n=0}^\8  e^{-n\beta}| \sqrt{n + 1} p_k(n) |^p  \right)^{1/p} & \leq (1 - e^{-\beta})^{ (p - 4)/(4p) } \left( \sum_{n=0}^\8  e^{-n\beta}| \sqrt{n + 1} p_k(n) |^4  \right)^{1/4} \\
			& \leq  (1 - e^{-\beta})^{ -1/4 } C_5(\beta)^{(2k + 1)/2}\sk{ \sum_{n=0}^\8 e^{-n\beta} (n + 1)|p_k(n)|^2 }^{1/2} \\
			& \leq \mk{ (1 - e^{-\beta})^{ -1/4 } C_5(\beta) }^{ (2k + 1)/2 }\sk{ \sum_{n=0}^\8 e^{-n\beta} (n + 1)|p_k(n)|^2 }^{1/2} \\
			& \leq \mk{ (1 - e^{-\beta})^{ -1/4 } C_5(\beta)p/2 }^{ (2k + 1)/2 }\sk{ \sum_{n=0}^\8 e^{-n\beta} (n + 1)|p_k(n)|^2 }^{1/2}.
		\end{align*}
		Note that $ (1 - e^{-\beta})^{ -1/4 }/2 \leq C_1(\beta) $, we take 
		\[
		C_3(\beta) = C_1(\beta)C_5(\beta).
		\]
		This completes the proof.
	\end{proof}

	\begin{rem}
		In Lemma \ref{Main-m=1}, the order $p^{(2k+1)/2}$ is optimal if we take $p_k(n)=n^k$ for any $n\in\mathbb{N}$.
	\end{rem}

	Now we are ready to prove Theorem \ref{Main-Lem} for all $k \in \nat$, $-k \leq m \leq k$. Define the function space
	\begin{equation*}
		F_{k, m} := \lk{ f: \nat \to \com: f(n) = \sum_{ \substack{2i+m\leq k \\ i+m\geq 0, 0\leq i \leq n} } c_{i, i+m}d_{n, i} d_{n+m, i+m}, \  c_{i, i+m}\in \mathbb{C}}. 
	\end{equation*}
	We need the following auxiliary lemma.
	\begin{lemma}
		Let $k \geq 1$ and $0 \leq m \leq k - 1$. For any $f \in F_{k, m + 1}$, there exists a function $g \in F_{k - 1, m}$ such that
		\begin{equation}\label{F_km-1}
			f(n) = (n + m + 1)^{1/2} g(n).
		\end{equation}
		Moreover, for any $f \in F_{k - 1, m + 1}$, there exists a function $g \in F_{k, m}$ such that
		\begin{equation}\label{F_km-2}
			(n + 1)^{1/2}f(n) = g(n + 1).
		\end{equation}
	\end{lemma}
	
	\begin{proof}
		We first prove \eqref{F_km-1}. For any $f \in F_{k, m + 1}$, we have
		\[
		f(n) = \sum_{ \substack{2i+(m + 1)\leq k \\ i+(m + 1)\geq 0, 0\leq i \leq n} } c_{i, i+m+1}d_{n, i} d_{n+(m + 1), i+(m + 1)}.
		\]
		Note that
		\begin{align*}
			d_{n+m+1, i+m+1} & = \sk{ \frac{ (n + m + 1)! }{ (n - i)! } }^{1/2} = (n + m + 1)^{1/2} \sk{ \frac{ (n + m)! }{ (n - i)! } }^{1/2} \\
			& = (n + m + 1)^{1/2} d_{n + m, i + m}.
		\end{align*}
		Thus
		\begin{align*}
			f(n) & = (n + m + 1)^{1/2} \sum_{ \substack{2i+(m + 1)\leq k \\ i+(m + 1)\geq 0, 0\leq i \leq n} } c_{i, i+m+1}d_{n, i} d_{n+m, i+m } \\
			& = (n + m + 1)^{1/2} \sum_{ \substack{2i+m\leq k - 1 \\ i+m \geq 0, 0\leq i \leq n} } c_{i, i+m+1}d_{n, i} d_{n+m, i+m } \\
			& =: (n + m + 1)^{1/2}g(n),
		\end{align*}
		which verifies the first assertion. Thus we obtain \eqref{F_km-1}.
		
		Now we prove the second assertion. For any $f \in F_{k - 1, m + 1}$, we have
		\[
		(n + 1)^{1/2} f(n) = (n + 1)^{1/2} \sum_{ \substack{2i + (m + 1) \leq k - 1 \\ i + (m + 1) \geq 0 \\ 0 \leq i \leq n} } c_{i, i+m+1}d_{n, i} d_{n+m+1, i+m+1}.
		\] 
		We observe that
		\begin{align*}
			(n + 1)^{1/2}d_{n, i} & = (n + 1)^{1/2} \sk{ \frac{n!}{ (n - i)! } }^{1/2} = \sk{ \frac{(n + 1)!}{ (n + 1 - (i + 1))! } }^{1/2} = d_{n + 1, i + 1}.
		\end{align*}
		Hence
		\begin{align*}
			(n + 1)^{1/2}	f(n) & = \sum_{ \substack{2i + (m + 1) \leq k - 1 \\ i + (m + 1) \geq 0 \\ 0 \leq i \leq n} } c_{i, i+m+1}d_{n + 1, i + 1} d_{n+m+1, i+m+1} \\
			& =  \sum_{ \substack{2(i + 1) + m  \leq k \\ (i + 1) + m \geq 0 \\ 1 \leq i + 1 \leq n + 1} } c_{i, i+m+1}d_{n + 1, i + 1} d_{n+1 + m, i+1 + m} \\
			& = \sum_{ \substack{2j + m  \leq k \\ j + m \geq 0 \\ 1 \leq j \leq n + 1} } c_{j - 1, j+m}d_{n + 1, j} d_{n+1 + m, j + m} \\
			& = : g(n + 1).
		\end{align*}
		We set $g(0) = 0$, thereby $g \in F_{k, m}$. This verifies the second assertion \eqref{F_km-2}.
	\end{proof}
	
	Now we proceed to prove Theorem \ref{Main-Lem} for $k \in \nat$ and $0 \leq m \leq k$. We restate the following lemma, which is equivalent to Theorem \ref{Main-Lem} for $k \in \nat$ and $0 \leq m \leq k$.
	
	\begin{lemma}\label{lemma3.6}
		Let $k \in \nat$ be fixed. For all $0 \leq m \leq k$, we have
		\begin{equation}\label{P}
			\sk{ \sum_{n \geq 0} e^{-n \beta} | f_{k, n, m} |^p }^{1/p} \leq (2e^{\beta/2})^{m} [C_3(\beta)p]^{k/2} \sk{ \sum_{n \geq 0} e^{-n \beta} | f_{k, n, m} |^2 }^{1/2}. \tag{P$(k, m)$}
		\end{equation}
	\end{lemma}
	
	\begin{proof}
		We prove \eqref{P} by induction for the parameter $m$. For $m = 0$, by Lemma \ref{Main-m=0} and the fact that $C_1(\beta) \leq C_3(\beta)$, P$(k, 0)$ holds for all $k \geq 0$. For $m = 1$, it suffices to verify P$(k, 1)$ for all $k \geq 1$. By \eqref{F_km-1}, there exists $g \in F_{k - 1, 0}$ such that $f_{k, n, 1} = (n + 1)^{1/2}g(n)$. The statement then follows from Lemma \ref{Main-m=1}.
		
		Assume P$(k, m)$ holds for all $k \geq m \geq 1$. The induction target is to show that P$(l, m + 1)$ holds for all $l \geq m + 1$. The left-hand-side of P$(l, m + 1)$ reads
		\begin{align}\label{P(l,m)-eq1}
			\text{LHS of P$(l, m + 1)$} & \leq \sum_{n = 0}^{m - 1} e^{-n \beta/p} |f_{l, n, m+1}| + \sk{ \sum_{n \geq m} e^{-n \beta} |f_{l, n, m+1}|^p }^{1/p} \notag\\
			& =: J_1 + J_2.
		\end{align}
		For the term $J_1$, by the Cauchy-Schwartz inequality, we obtain
		\begin{align}\label{J_1}
			J_1 & \leq \sk{ \sum_{n = 0}^{m - 1} e^{ 2n\beta\sk{ 1/2 - 1/p } } }^{1/2} \sk{\sum_{n = 0}^{m - 1} e^{-n \beta} |f_{l, n, m+1}|^2 }^{1/2} \notag\\
			& \leq \sqrt{m} e^{(m - 1)\beta(1/2 - 1/p)} \sk{\sum_{n \geq 0} e^{-n \beta} |f_{l, n, m+1}|^2 }^{1/2} \notag\\
			& \leq (2e^{\beta/2})^m \mk{ C_3(\beta)p }^{l/2}\sk{\sum_{n \geq 0} e^{-n \beta} |f_{l, n, m+1}|^2 }^{1/2} .
		\end{align}
		For the term $J_2$, combining \eqref{F_km-1} and \eqref{F_km-2}, for any $f \in F_{l, m + 1}$, there exists $g \in F_{l, m - 1}$ such that
		\[
		f_{l, n, m + 1} = f(n) = \sk{ \frac{ n + m + 1 }{ n + 1 } }^{1/2} g(n + 1), \quad g(0) = 0.
		\]
		Then we have
		\begin{align}\label{J_2}
			J_2 & \leq \sk{ \sum_{n \geq m} e^{-n \beta}2^{p/2} | g(n + 1) |^p }^{1/p} = \sqrt{2} e^{\beta/p} \sk{ \sum_{n \geq m + 1} e^{-n \beta} | g(n) |^p }^{1/p} \notag\\
			(\text{by {\color{blue}P$(l, m-1)$}}) & \leq \sqrt{2} e^{\beta/p} (2e^{\beta/2})^{m-1} \mk{C_3(\beta) p}^{l/2} \sk{ \sum_{n \geq 0} e^{-n \beta} | g(n) |^2 }^{1/2} \notag\\
			(\text{since $g(0) = 0$}) & = e^{\beta/p} (2e^{\beta/2})^{m-1/2} \mk{C_3(\beta) p}^{l/2} e^{-3\beta/4} \sk{ \sum_{n \geq 0} e^{-n \beta} | g(n + 1) |^2 }^{1/2} \notag\\
			& \leq (2e^{\beta/2})^{m} \mk{C_3(\beta) p}^{l/2}\sk{ \sum_{n \geq 0} e^{-n \beta} \jdz{ \sk{ \frac{ n + m + 1 }{ n + 1 } }^{1/2} g(n + 1) }^2 }^{1/2} \notag\\
			& = (2e^{\beta/2})^{m} \mk{C_3(\beta) p}^{l/2} \sk{ \sum_{n \geq 0} e^{-n \beta} |f_{l, n, m + 1}|^2 }^{1/2} .
		\end{align}
		Substituting \eqref{J_1} and \eqref{J_2} into \eqref{P(l,m)-eq1} yields
		\begin{align*}
		\text{LHS of P$(l, m + 1)$} & \leq 2(2e^{\beta/2})^{m} \mk{C_3(\beta) p}^{l/2} \sk{ \sum_{n \geq 0} e^{-n \beta} |f_{l, n, m + 1}|^2 }^{1/2}  \\
		& \leq (2e^{\beta/2})^{m + 1} \mk{C_3(\beta) p}^{l/2} \sk{ \sum_{n \geq 0} e^{-n \beta} |f_{l, n, m + 1}|^2 }^{1/2}. 
		\end{align*}
		This proves the induction target P$(l, m + 1)$.
	\end{proof}
	
	In order to pass Lemma \ref{lemma3.6} to the case $-k \leq m < 0$, the following lemma is necessary.
	\begin{lemma}
		Let $k \geq 1$ and $-k \leq m < 0$. For any $f \in F_{k, m}$, there exists a function $g \in F_{k, -m}$ such that
		\begin{equation}\label{Lem-m<0}
			f(n) = \begin{cases}
				g(n + m), & \text{if } \quad n \geq -m, \\
				0, & \text{if } \quad 0 \leq n < -m.
			\end{cases} 
		\end{equation}
	\end{lemma}
	\begin{proof}
		For any $f \in F_{k, m}$, $f(n) = 0$ if $0 \leq n < -m$. For $n \geq -m$, we have
		\begin{align*}
			f(n) & =  \sum_{ \substack{2i+ m \leq k \\ i+m\geq 0, 0\leq i \leq n} } c_{i, i+m}d_{n, i} d_{n+m , i+m} = \sum_{ \substack{2i+ m \leq k \\ -m \leq i \leq n} } c_{i, i+m}d_{n, i} d_{n+m , i+m} \\
			& = \sum_{ \substack{2(i + m) + (-m) \leq k \\ 0 \leq i + m \leq n + m} } c_{i, i+m}d_{(n + m) + (-m), (i + m) + (-m)} d_{n+m , i+m} =: g(n + m).
		\end{align*}
		It is easy to see that $g \in F_{k, -m}$. We complete the proof.
	\end{proof}
	
	We now come to the proof of Theorem \ref{Main-Lem}.
	\begin{proof}[Proof of Theorem \ref{Main-Lem}]
		By \eqref{P}, we have for $0 \leq m \leq k$,
		\begin{align*}
			\sk{ \sum_{n \geq 0} e^{-n \beta} | f_{k, n, m} |^p }^{1/p} & \leq e^{m\beta (1/(2p) - 1/4) } \mk{ 2e^{\beta/2}e^{\beta \sk{ 1/4 - 1/(2p) } } }^m [C_3(\beta)p]^{k/2} \sk{ \sum_{n \geq 0} e^{-n \beta} | f_{k, n, m} |^2 }^{1/2} \\
			& \leq e^{m\beta (1/(2p) - 1/4) } \mk{ 4e^{3\beta/2}C_3(\beta) p }^{k/2} \sk{ \sum_{n \geq 0} e^{-n \beta} | f_{k, n, m} |^2 }^{1/2}. 
		\end{align*}
		Hence it suffices to choose
		\begin{equation}\label{C-cond}
			C(\beta) = 4e^{3\beta/2}C_3(\beta).
		\end{equation}
		For $-k \leq m < 0$, if $f \in F_{k, m}$, there exists a function $g \in F_{k, -m}$ such that \eqref{Lem-m<0} holds. Then
		\begin{align*}
			\sk{ \sum_{n \geq 0} e^{-n \beta} | f_{k, n, m} |^p }^{1/p} & = \sk{ \sum_{n \geq -m} e^{-n \beta} | g(n + m) |^p }^{1/p} \\
			& = e^{m\beta/p}\sk{ \sum_{n \geq 0} e^{-n \beta} | g(n) |^p }^{1/p} \\
			(\text{by {\color{blue}P$(k, -m)$}}) & \leq e^{m\beta/p} (2e^{\beta/2})^{-m} \mk{ C_3(\beta)p }^{k/2} \sk{ \sum_{n \geq 0} e^{-n \beta} | g(n) |^2 }^{1/2} \\
			& = e^{m\beta/p} (2e^{\beta/2})^{-m} \mk{ C_3(\beta)p }^{k/2} e^{-m\beta/2} \sk{ \sum_{n \geq 0} e^{-n \beta} | f_{k, n, m} |^2 }^{1/2} \\
			& = e^{ m\beta \sk{ 1/(2p) - 1/4 } } \sk{ 4e^{\beta \sk{3/2 - 1/p} } }^{-m/2} \mk{ C_3(\beta)p }^{k/2} \sk{ \sum_{n \geq 0} e^{-n \beta} | f_{k, n, m} |^2 }^{1/2} \\
			& \leq e^{ m\beta \sk{ 1/(2p) - 1/4 } } \mk{ 4 e^{3\beta/2}C_3(\beta)p }^{k/2} \sk{ \sum_{n \geq 0} e^{-n \beta} | f_{k, n, m} |^2 }^{1/2}.
		\end{align*}
		Under the condition \eqref{C-cond}, we obtain the desired estimate.
	\end{proof}
	
	\bigskip
	
		\section{Proof of Theorem \ref{sub-MainThm}}\label{Section5}
		
	\subsection{Upper bounde of Optimal time $t_p$} This subsection is devoted to the proof of
	$$  e^{2 \tau t_p} \lesssim_\beta (p-1).  $$
	Since $\rho^{1/2} B(H) \rho^{1/2}$ is dense in $L_2(\rho)$, it suffices to show that if $e^{2 \tau t} \gtrsim_\beta (p-1)$, then for any $ x\in B(H) $,
	\begin{equation}\label{new-equ5.1}
		\| T_t^{(2)}(\rho^{1/2}x\rho^{1/2}) \|_{L_p(\rho)} \leq \|\rho^{1/2}x\rho^{1/2}\|_{L_2(\rho)}.
	\end{equation}
	We start the proof.
	 
	\begin{proof}
		Assume that $e^{2 \tau t} \geq d(\beta)^{-2}p$ where $d(\beta)$ is as in Lemma \ref{new-lemma3.4}. By standard approximation, without loss of generality, we can assume that $\{c_{n, m}: m, n=0, 1, 2, \cdots\}\subset \mathbb{C}$ has finite non-zero terms, and suppose
		\[  
		x=\rho^{1/4}y\rho^{1/4}=\sum_{m, n=0}^{\8} c_{m, n} \xi_{m, n}, \quad  \| \rho^{1/4}y\rho^{1/4} \|_2 = \sk{ \sum_{m, n=0}^{\8} |c_{m, n}|^2 }^{1/2}.  
		\]
		Then by \eqref{Eigen}, one has
		\[
		T_t^{(2)}(\rho^{1/2}y\rho^{1/2})=\rho^{1/4}P_t(\rho^{1/4}y\rho^{1/4})\rho^{1/4}=\sum_{m, n=0}^{\8} c_{m, n} e^{-t(m \tau_1+n \tau_2)} \rho^{1/4}\xi_{m, n}\rho^{1/4}.   
		\]
		For $k\geq 1$, define
		\[
		y_k= \sum_{ m+n=k} c_{m, n}  e^{-t(m \tau_1+n \tau_2)}  \xi_{m, n}.
		\]
		Then one has
		$$  \|y_k\|_2 =\left(\sum_{ m+n=k} |c_{m, n}|^2  e^{-2t(m \tau_1+n \tau_2)}\right)^{1/2}\leq e^{-tk\tau} \left(\sum_{ m+n=k} |c_{m, n}|^2 \right)^{1/2}.  $$
		Note that $e^{-t\tau}\leq d(\beta) p^{-1/2}$, and by Lemma \ref{new-lemma3.4}, the left-hand-side of \eqref{new-equ5.1} can be estimated by
		\begin{align*}
			\| T_t^{(2)}(\rho^{1/2}y\rho^{1/2}) \|_{L_p(\rho)} & = \norm{  \rho^{1/(2p)-1/4} \left( c_{0,0} \rho^{1/2} + \sum_{k=1}^\infty y_k  \right)  \rho^{1/(2p)-1/4} }_p \\
			& \leq \sk{ |c_{0,0}|^2+\sum_{k=1}^{\8} \sum_{m+n=k} |c_{m, n}|^2 }^{1/2}\\
			& =  \| \rho^{1/2} y \rho^{1/2} \|_{L_2(\rho)}.
		\end{align*}
		This completes the proof.
	\end{proof}
	
	\subsection{Lower bound of Optimal time $t_p$}\label{Section4}
	In this subsection, we discuss the lower bound of the optimal time $t_p$ for the hypercontractivity inequalities associated with the quantum Ornstein-Uhlenbeck semigroups. We start with the following auxiliary lemma.
	\begin{lemma}
		Let $1 \leq p < \infty$. Assume that $(a_n)_{n \geq 0} \in \ell_p(\nat)$, and define
		$$ A = \sum_{n \geq 0} a_n \sk{ e_n \otimes e_{n + 1} + e_{n + 1} \otimes e_n } \in B(\ell_2(\nat)). $$
		Then
		\begin{equation}\label{Schatten-norm}
			\frac{ 2^{1/p} }{3}\sk{ \sum_{n \geq 0} |a_n|^p }^{1/p} \leq \| A \|_p \leq 2\sk{ \sum_{n \geq 0} |a_n|^p }^{1/p}.
		\end{equation}
	\end{lemma}
	\begin{proof}
		The second inequality easily follows from the triangle inequality. For the first inequality, we define $P_n$ as the projection of $\ell_2(\nat)$ onto ${\rm span}\lk{ e_n, e_{n + 1} }$. Then we have
		\[
		P_{n} A P_{n} (e_k) = \begin{cases}
			a_n e_{n + 1}, & \text{if } k = n, \\
			a_n e_{n}, & \text{if } k = n + 1, \\
			0, & \text{otherwise.}
		\end{cases}
		\]
		Hence 
		\[
		\| P_n A P_n \|_p = 2^{1/p}|a_n|.
		\]
		Consider the family $\lk{ P_{3n}AP_{3n} : n \in \nat }$. We observe that $P_{3i}AP_{3j} \neq 0$ if and only if $i = j$, hence
		\[
		\norm{ \sk{ \sum_{n \geq 0} P_{3n} } A \sk{ \sum_{n \geq 0} P_{3n} } }_p = \left(\sum_{n \geq 0} \| P_{3n}AP_{3n} \|_p^p \right)^{1/p} = 2^{1/p}\left(\sum_{n \geq 0} |a_{3n}|^p\right)^{1/p}.
		\]
		We adopt the same argument to the family $\lk{ P_{3n + 1}AP_{3n + 1} : n \in \nat }$ and $\lk{ P_{3n + 2}AP_{3n + 2} : n \in \nat }$, and we obtain
		\[
		\sum_{i = 0}^2\norm{ \sk{ \sum_{n \geq 0} P_{3n + i} } A \sk{ \sum_{n \geq 0} P_{3n + i} } }_p \geq 2^{1/p} \sk{ \sum_{n \geq 0} |a_n|^p }^{1/p}.
		\]
		Note that
		\[
		\norm{ \sk{ \sum_{n \geq 0} P_{3n + i} } A \sk{ \sum_{n \geq 0} P_{3n + i} } }_p \leq \norm{  \sum_{n \geq 0} P_{3n + i} }_\8 \| A \|_p \norm{ \sum_{n \geq 0} P_{3n + i} }_\8 \leq \| A \|_p. 
		\]
		Combining these estimates yields the desired inequality.
	\end{proof}
	
	We restate Theorem \ref{sub-MainThm} on the estimate of the optimal time $t_p$.
	\begin{thm}\label{Optimal}
		Let $p\geq 2$. There exist positive constants $\wt{c}(\beta)$, $\wt{C}(\beta)$ such that,
		\[
		\wt{c}(\beta)(p-1) \leq e^{2\tau t_p} \leq \wt{C}(\beta) (p - 1).
		\]
		More precisely,
		\[
		\wt{c}(\beta) = \frac{ \min\lk{1, \beta^{-2}} }{ 144e}, \quad \wt{C}(\beta) = 2d(\beta)^{-2},
		\]
		where $d(\beta)$ is the constant in Lemma \ref{new-lemma3.4}.
	\end{thm}
	
	\begin{proof}
		We first prove the upper bound. Combining Lemma \ref{Prop2.1}, Proposition \ref{Prop2.2}, Theorem \ref{Main-Lem} and Lemma \ref{new-lemma3.4}, we conclude that the optimal time $t_p$ satisfies 
		\begin{equation}\label{Upper-1}
			e^{2\tau t_p} \leq d(\beta)^{-2}p \leq 2d(\beta)^{-2}(p-1).
		\end{equation}
		For the lower bound, without loss of generality, we assume $\tau_1 \leq \tau_2$, then $\tau = \tau_1$. 
		We can take $ \rho^{1/4} x \rho^{1/4} = \xi_{1, 0} =  A_1^* \xi_0 $. Then $\rho^{1/2}x\rho^{1/2} \in L_2(\rho)$ and
		\[
		T_{t_p}^{(2)}( \rho^{1/2} x \rho^{1/2} ) = e^{-\tau t_p} \rho^{1/4} A_1^*\xi_0 \rho^{1/4}, \quad \| \rho^{1/2} x \rho^{1/2} \|_{L_2(\rho)} = \| \rho^{1/4} x \rho^{1/4} \|_2 = 1.
		\]
		Note that one has
		\begin{equation}\label{Lower-1}
			\| e^{-\tau t_p} \rho^{1/(2p) - 1/4} A_1^*\xi_0 \rho^{1/(2p) - 1/4} \|_p =	\| T_{t_p}^{(2)}( \rho^{1/2} x \rho^{1/2} ) \|_{L_p(\rho)} \leq \| \rho^{1/2} x \rho^{1/2} \|_{L_2(\rho)} = 1.
		\end{equation}
		By definition, we can write $A_1^*\xi_0$ as a matrix of the form
		\begin{align*}
			A_1^*\xi_0 & = \frac{1}{\sqrt{2}} (D_1^* - D_2^*)\rho^{1/2} \\
			& = \frac{ \sk{ 2 \sinh(\beta/2) }^{-1/2} }{\sqrt{2}} \sk{ e^{\beta/4} - e^{-3\beta/4} }  \sk{ a^*\rho^{1/2} + \rho^{1/2}a } \\
			& = \frac{\sk{ 2\sinh(\beta/2) }^{-1/2}}{\sqrt{2}} \sk{ e^{\beta/4} - e^{-3\beta/4} }  \sum_{n = 0}^\infty \frac{ e^{-n \beta/2}\sqrt{n + 1} }{ \sk{ 1 - e^{-\beta} }^{-1/2} } \sk{ e_n \otimes e_{n + 1} + e_{n + 1} \otimes e_n } \\
			& = \frac{1-e^{-\beta}}{\sqrt{2}} \sum_{n = 0}^\infty e^{-n \beta/2}\sqrt{n + 1} \sk{ e_n \otimes e_{n + 1} + e_{n + 1} \otimes e_n } .
		\end{align*}
		For convenience, we denote
		\[
		\quad Z = e^{-\tau t_p} \rho^{1/(2p) - 1/4} A_1^*\xi_0 \rho^{1/(2p) - 1/4}.
		\]
		Hence
		$$  Z=e^{-\tau t_p} 
		\frac{ (1 - e^{-\beta})^{1/p + 1/2} }{\sqrt{2} } e^{-\beta(1/(2p) - 1/4)} \sum_{n = 0}^\infty \sqrt{n + 1} e^{-n\beta/p } \sk{ e_n \otimes e_{n + 1} + e_{n + 1} \otimes e_n }. $$
		However, form \eqref{Lower-1} and by \eqref{Schatten-norm},
		\begin{align}\label{Z-p-norm}
			1 & \geq \| Z \|_p \geq e^{-\tau t_p} \frac{ 2^{1/p}\sk{ 1 - e^{-\beta} }^{1/2+1/p} }{3\sqrt{2}} \sk{ \sum_{n = 0}^\infty (n + 1)^{p/2} e^{-n\beta }  }^{1/p} \notag\\
			& \geq e^{-\tau t_p} \frac{ 2^{1/p}\sk{ 1 - e^{-\beta} }^{1/2+1/p} }{3\sqrt{2}} \beta^{-(1/2 + 1/p)} \sk{ \frac{ \pi p}{4} }^{1/(2p)} \sk{ \frac{p}{2e} }^{1/2} \notag\\
			& \geq e^{-\tau t_p} \frac{ [\beta^{-1}(1 - e^{-\beta})  ]^{1/2 + 1/p} }{6 \sqrt{e}} p^{1/2}.
		\end{align}
		Note that
		\[ 
		\sk{\frac{ 1 - e^{-\beta} }{ \beta } }^{1/2 + 1/p} \geq \frac{1}{2} \min\lk{ 1, \beta^{-1} }. 
		\]
		This implies
		$$  e^{\tau t_p} \geq \frac{ \min\lk{ 1, \beta^{-1} } }{12 \sqrt{e}} p^{1/2}.  $$
		Hence, the proof is completed.
	\end{proof}
	
	\begin{rem}
		In the proof of Theorem \ref{Optimal}, one can usually take $x= \xi_0 + \varepsilon \xi_{1, 0}$ or $x= \xi_0 + \varepsilon \xi_{0, 1}$ as the test operators. But it also gives the same estimate as in Theorem \ref{Optimal}. However, the calculation becomes significantly more complicated. So for the simplicity of presentation, we choose $x= \varepsilon \xi_{1, 0}$ or $x= \varepsilon \xi_{0, 1}$ as the test operators.
		\end{rem}

	\bigskip {\textbf{Acknowledgments.}} We thank Marius Junge and Li Gao for helpful discussions and comments. We also thank Quanhua Xu, \'{E}ric Ricard, Hun Hee Lee and Sang-Gyun Youn for useful suggestions to help us improve the presentation of this paper. We thank ChatGPT-6 Astra for helping us verify the calculation in Lemma 3.4. 
	
	

\end{document}